\documentclass[12pt,centertags,oneside]{article}
\usepackage{amsmath,amstext,amsthm,amscd,typearea}
\usepackage{amssymb}
\usepackage{a4wide}
\usepackage[mathscr]{eucal}
\usepackage{mathrsfs}
\usepackage{typearea}
\usepackage{charter}
\usepackage{pdfsync}
\usepackage{url}

\usepackage{xcolor}

\usepackage[a4paper,width=16.2cm,top=3cm,bottom=3cm]{geometry}

\numberwithin{equation}{section}

\newtheorem{theorem}{Theorem}[section]
\newtheorem{definition}[theorem]{Definition}
\newtheorem{proposition}[theorem]{Proposition}
\newtheorem{corollary}[theorem]{Corollary}
\newtheorem{lemma}[theorem]{Lemma}
\newtheorem{remark}[theorem]{Remark}

\newcommand{\cali}[1]{\mathscr{#1}}

\newcommand{\vol}{\mathop{\mathrm{vol}}}

\newcommand{\ddc}{dd^c}
\newcommand{\dc}{d^c}

\newcommand{\Sing}{\text{\normalfont Sing}}

\newcommand{\C}{\mathbb{C}}

\newcommand{\N}{\mathbb{N}}

\renewcommand\P{\mathbb{P}}

\title{\bf  Loss of mass of non-pluripolar products}
\providecommand{\keywords}[1]{\textbf{\textit{Keywords:}} #1}
\providecommand{\subject}[1]{\textbf{\textit{Mathematics Subject Classification 2010:}} #1}

\author{Duc-Viet Vu}

\newcommand{\Addresses}{{
		\bigskip
		\footnotesize
		\textsc{Duc-Viet Vu, University of Cologne, Division of Mathematics, Department of Mathematics and Computer Science, Weyertal 86-90, 50931, K\"oln,  Germany}
		\noindent
		\par\nopagebreak
		\noindent
		\textit{E-mail address}: \texttt{vuduc@math.uni-koeln.de}	
}}

\date{\today}
\begin{document}
\maketitle
\begin{abstract} It is a well-known fact that the non-pluripolar self-products of a closed positive $(1,1)$-current in a big nef cohomology class on a compact K\"ahler manifold are not of full mass in the presence of positive Lelong numbers of the current in consideration. In this paper, we give a quantitative version of the last property.  Our proof involves a generalization of Demailly's comparison of Lelong numbers to the setting of theory of density currents, a reversed Alexandrov-Fenchel inequality for non-pluripolar products and the notion of relative non-pluripolar products.    
\end{abstract}
\noindent
, {density current}, {tangent current}, {Lelong number}, {full mass intersection}, {Alexandrov-Fenchel inequality}.
\\

\noindent
\subject{32U15}, {32Q15}.

\section{Introduction}

Let $X$ be a compact K\"ahler manifold of dimension $n$. Let $\omega$ be a K\"ahler form on $X$. For every closed positive $(p,p)$-current $S$ on $X$, we recall that the mass $\|S\|$ of $S$ is given by $\int_X S \wedge \omega^{n-p}$.  A cohomology $(p,p)$-class $\alpha$ is said to be \emph{pseudoeffective} if it is the class of a closed positive $(p,p)$-current. For  pseudoeffective $(p,p)$-classes $\alpha_1$ and $\alpha_2$ on $X$, we write $\alpha_1 \le  \alpha_2$ if $\alpha_2- \alpha_1$ is pseudoeffective. For a pseudoeffective $(p,p)$-class $\alpha$, we  use the notation $\|\alpha\|$ to denote the norm of a closed positive current $S$ representing $\alpha$. This quantity is independent of the choice of $S$. 

Let $1\le m \le n$ be an integer. Let $\alpha_1, \ldots, \alpha_m$ be  big nef cohomology classes on $X$. Let $T_j$ be a closed positive current of bi-degree $(1,1)$ in $\alpha_j$ for $1 \le j \le m$. The non-pluripolar product $\langle T_1 \wedge \cdots \wedge T_m \rangle$ plays an important role in complex geometry; see \cite{BT_fine_87,BEGZ,Lu-Darvas-DiNezza-logconcave,GZ-weighted,Viet-generalized-nonpluri} and references therein. It generalizes the classical product of $(1,1)$-currents of bounded potentials (see \cite{Bedford_Taylor_82}). A key phenomenon about that notion is that the non-pluripolar products don't preserve the mass, \emph{i.e,} in general, we have 
\begin{align}\label{ine-sosanhmassT1denTmomega1m}
\big \| \langle T_1 \wedge \cdots \wedge T_m \rangle \big \| \le \big \| \alpha_1 \wedge \cdots \wedge \alpha_m \big \|.
\end{align}
We indeed have a much stronger property that the cohomology class of the current  $\langle T_1 \wedge \cdots \wedge T_m \rangle$ is less than or equal to $\alpha_1 \wedge \cdots \wedge \alpha_m$, see \cite[Theorem 1.1]{Viet-generalized-nonpluri} and also  \cite{BEGZ,Lu-Darvas-DiNezza-mono,WittNystrom-mono} for the case where $m=n$.

When the equality in (\ref{ine-sosanhmassT1denTmomega1m}) occurs, the currents $T_1,\ldots,T_m$ are said to be \emph{of full mass intersection}. The last notion is at the heart of the theory of non-pluripolar products. Characterizing such currents is hence important. So far we have known  that the positivity of Lelong numbers of $T_j$'s is an obstruction to being of full mass intersection; see  \cite{Lu-Darvas-DiNezza-singularitytype,GZ-weighted,Vu_lelong-bigclass} and references therein. We are interested in understanding this property from a quantitative point of view. Although we think that this question is worth studying, there has not been much research in this direction.  To our best knowledge, the first available result is probably \cite[Theorem 1.2]{Vu_lelong-bigclass} which treated the case when $m=n$; see also \cite[Corollary 7.6]{Demailly_regula_11current} for a related result.  Unfortunately, the method in \cite{Vu_lelong-bigclass} is not good enough to obtain appropriate quantitative estimates when $m<n$, see comments after Theorem \ref{the-self-intersec} below. The following is our main result giving such an quantitative estimate for every $m$ in the case of self-intersection.   

\begin{theorem} \label{the-self-intersec} Let $\cali{N}_0$ be a compact subset in the intersection of the big and nef cones of $X$.  Let $\alpha$ be a cohomology $(1,1)$-class in $\cali{N}_0$. Let $T$ be a closed positive current in $\alpha$. Let $1\le m \le n$ be an integer.  Let $\cali{V}$ be  the set of  maximal irreducible analytic subsets $V$ of dimension at least $n-m$ of $X$ such that the generic Lelong number $\nu(T,V)$ of $T$ along $V$ is strictly positive. Then, there exists a constant $C>0$ independent of $\alpha$ such that  
\begin{align}\label{ine-obstructionselfinter}
\big\| \alpha^m - \{\langle T^m \rangle \}  \big\| \ge C \bigg(\sum_{V \in \cali{V}}\nu(T,V)^{n- \dim V} \vol(V) \bigg)^{2^m}.
\end{align}
\end{theorem}

Here we put $\vol(V):= \int_V \omega^{\dim V}$.  The \emph{maximality} of $V$ in the above result means that in the set of  irreducible analytic subsets $V'$ of $X$ such that $\nu(T,V')>0$, we consider the partial order given by the inclusion of sets, and $V$ is maximal if it is so with respect to that order.  Theorem \ref{the-self-intersec} no longer holds if one enlarges $\cali{V}$ to include non-maximal analytic sets. One can see it as follows: take for example $T$ to be a current with analytic singularities along analytic set $V_0$ of dimension $>n-m$, for a suitable choice of $V_0$ (e.g, $V_0$ is biholomorphic to a complex projective space), we see that the volume of an irreducible set $V$ of dimension $n-m$ in $V_0$ can be as big as we want. 

  
To get motivated about Theorem \ref{the-self-intersec}, one can consider an ideal classical example where $T$ has analytic singularities along  an irreducible analytic set $V$ of dimension $n-m$, in this case, the right-hand side of (\ref{ine-obstructionselfinterkahler}) can be replaced by $\nu(T,V)^m \vol(V)$. When  $\cali{N}_0$ is compact in the K\"ahler cone of $X$, we have a more precise estimate; see Theorem \ref{the-kahlerself-intersec} for details. 

We underline that  the arguments used in \cite[Theorem 1.2]{Vu_lelong-bigclass} and  \cite[Theorem 1.2]{Viet-generalized-nonpluri} are not sufficient to get (\ref{ine-obstructionselfinter})  because the proofs there use desingularizations of $V$ and the blowup along $V$; hence the lower bounds obtained there depend intrinsically on $V$.   We don't know what should be an optimal lower bound for the left-hand side of (\ref{ine-obstructionselfinter}).  

As a direct consequence of Theorem \ref{the-self-intersec},  we infer that if $\{\langle T^m \rangle \}= \alpha^m$, then for every (not necessarily maximal) irreducible analytic subset $V$ of dimension at least $n-m$, we have $\nu(T,V)=0$. This recovers a previous known result in \cite{Vu_lelong-bigclass} (see Theorem 1.3 there). As a byproduct of our method, we also generalize some results from \cite{Lu-Darvas-DiNezza-singularitytype}, see Corollaries \ref{cor-Tcongvoitheta} and \ref{cor-Tcongvoitheta2} below. \emph{Finally we note that Theorem \ref{the-self-intersec} is new even when  $m=n$ and $\alpha$ is K\"ahler.} 

In order to prove Theorem \ref{the-self-intersec}, we will prove a reserved Alexandrov-Fenchel inequality for non-pluripolar products, see Proposition \ref{pro-theoremmain} below. This will reduce the question to the K\"ahler case. Next, we establish a generalization of Demailly's lower bound for Lelong numbers of intersection of currents to the setting of the theory of density currents introduced by Dinh-Sibony in \cite{Dinh_Sibony_density}, see Corollary \ref{cor-sosanhlelong}. This is a key in our treatment of the K\"ahler case. Moreover, we emphasize that the notion of relative non-pluripolar products in \cite{Viet-generalized-nonpluri} will play a crucial role in our proof. It serves as a bridge from density currents to non-pluripolar products. 


In the next section, we prove the above-mentioned lower bound for Lelong numbers of density currents. In Section \ref{sec-AF}, we establish a reversed Alexandrov-Fenchel type inequality for non-pluripolar products. Theorems \ref{the-self-intersec} is proved in Sections \ref{sec-the1}. 
\\

\noindent
\textbf{Acknowledgments.} The author would like to thank Tien-Cuong Dinh and Nessim Sibony  for fruitful discussions. This research  is supported by a postdoctoral fellowship of the Alexander von Humboldt Foundation.

\section{Lelong numbers of intersection of currents} \label{sec-lelong-density}

We first recall some basic properties of density currents. The last notion was introduced in \cite{Dinh_Sibony_density}. 

Let $X$ be a complex K\"ahler manifold of dimension $n$ and $V$ a smooth complex submanifold of $X$ of dimension $l.$  Let $T$ be a closed positive $(p,p)$-current on $X,$ where $0 \le p \le n.$  Denote by $\pi: E\to V$ the normal bundle of $V$ in $X$ and $\overline E:= \P(E \oplus \C)$ the projective compactification of $E.$ By abuse of notation, we also use $\pi$ to denote the natural projection from $\overline E$ to $V$. 

Let $U$ be an open subset of $X$ with $U \cap V \not = \varnothing.$  Let $\tau$ be  a smooth diffeomorphism  from $U$ to an open neighborhood of $V\cap U$ in $E$ which is identity on $V\cap U$ such that  the restriction of its differential $d\tau$ to $E|_{V \cap U}$ is identity.  Such a map is called \emph{an admissible map}.  Note that in \cite{Dinh_Sibony_density}, to define an admissible map,  it is required furthermore that $d\tau$ is $\C$-linear at every point of $V$. This difference doesn't affect what follows.  When $U$ is a small enough tubular neighborhood of $V,$ there always exists an admissible map $\tau$ by \cite[Lemma 4.2]{Dinh_Sibony_density}. In general, $\tau$ is not holomorphic.  When $U$ is a small enough local chart, we can choose a holomorphic admissible map by using suitable holomorphic coordinates on $U$.   For $\lambda \in \C^*,$ let $A_\lambda: E \to E$ be the multiplication by $\lambda$ on fibers of $E.$  We recall the following crucial result.

\begin{theorem} \label{th-dieukienHVconictangenetcurrent} (\cite[Theorem 4.6]{Dinh_Sibony_density})
Let $\tau$ be an admissible map defined on a tubular neighborhood of $V$. Then,  the family $(A_\lambda)_* \tau_* T$ is of mass uniformly bounded in $\lambda$, and  if $S$ is a limit current of the last family as $\lambda \to \infty$, then  $S$ is a current on $E$ which can be extended trivially through $\overline E \backslash E$ to be a closed positive current on $\overline E$  such that the cohomology class $\{S\}$ of $S$ in $\overline E$ is independent of the choice of $S$, and $\{S\}|_V= \{T\}|_V$,  and $\|S\| \le C \|T\|$ for some constant $C$ independent of $S$ and $T$. 
\end{theorem}

We call $S$  \emph{a tangent current to $T$ along $V$}, and its cohomology class is called \emph{the total tangent class of $T$ along $V$} and is denoted by $\kappa^V(T)$. By \cite[Theorem 4.6]{Dinh_Sibony_density} again,  if 
$$S=\lim_{k\to \infty} (A_{\lambda_k})_* \tau_* T$$ for some sequence $(\lambda_k)_k$ converging to $\infty$, then  for every open subset $U$ of $X$ and  every admissible map $\tau': U' \to E$ , we also have  
$$S=\lim_{k\to \infty} (A_{\lambda_k})_* \tau'_* T.$$
This is equivalent to saying that tangent currents are independent of the choice of the admissible map $\tau$.

\begin{definition} (\cite{Dinh_Sibony_density}) Let $F$ be a complex manifold and $\pi_F: F \to V$ a holomorphic submersion. Let $S$  be  a positive current of bi-degree $(p,p)$  on $F$. \emph{The h-dimension} of $S$ with respect to $\pi_F$ is the biggest integer $q$ such that $S \wedge \pi_F^* \theta^q \not =0$ for some Hermitian metric $\theta$ on $V$.  
\end{definition}

By a bi-degree reason, the h-dimension of $S$ is in $[\max\{l- p,0\}, \min\{\dim F -p,l\}]$. 
We have the following description of currents with minimal h-dimension. 

\begin{lemma} \label{le-minimalhdimension} (\cite[Lemma 3.4]{Dinh_Sibony_density})  Let $\pi_F: F \to V$ be a submersion. Let $S$ be a closed positive  current of bi-degree $(p,p)$ on $F$ of h-dimension $(l -p)$ with respect to $\pi_F$. Then $S= \pi^* S'$ for some closed positive current $S'$ on $V$. 
\end{lemma} 

By \cite{Dinh_Sibony_density}, the h-dimensions of tangent currents to $T$ along $V$ are the same and this number is called \emph{the tangential h-dimension of $T$ along $V$}.

Let $m\ge 2$ be an integer. Let $T_j$ be a closed positive current  of bi-degree $(p_j, p_j)$ for $1 \le j \le m$ on $X$ and let  $T_1 \otimes \cdots \otimes T_m$ be the tensor current of $T_1, \ldots, T_m$ which is a current on $X^m.$  A \emph{density current} associated to $T_1, \ldots,  T_m$ is a tangent current to $\otimes_{j=1}^m T_j$ along the diagonal $\Delta_m$ of $X^m.$ Let $\pi_m: E_m \to \Delta$ be the normal bundle of $\Delta_m$ in $X^m$. Denote by $[V]$ the current of integration along $V$.  When $m=2$ and $T_2 =[V]$, the density currents of $T_1$ and  $T_2$ are naturally identified with the  tangent currents to $T_1$ along $V$ (see \cite[Lemma 2.3]{Vu_density-nonkahler}).

The unique cohomology class of density currents associated to $T_1,\ldots,T_m$ is called \emph{the total density class of $T_1, \ldots, T_m$}. We denote the last class by $\kappa(T_1,\ldots, T_m)$. The tangential h-dimension of  $T_1 \otimes \cdots \otimes T_m$ along $\Delta_m$ is called \emph{the density h-dimension} of  $T_1, \ldots T_m$. 

\begin{lemma} \label{le-classDSproduct} (\cite[Section 5]{Dinh_Sibony_density})  Let $T_j$ be a closed positive current of bi-degree $(p_j,p_j)$ on $X$ for $1 \le j \le m$ such that $\sum_{j=1}^m p_j \le n$.  Assume that the density h-dimension of $T_1, \ldots, T_m$ is minimal, \emph{i.e}, equal to $n- \sum_{j=1}^m p_j$. Then the total density class of  $T_1, \ldots, T_m$ is equal to $\pi_{m}^*(\wedge_{j=1}^m\{T_j\})$. 
\end{lemma}

We recall the following result. 

\begin{theorem} \label{th-sosanhVV'densityDS} (\cite[Proposition 4.13]{Dinh_Sibony_density}) Let  $V'$ be a submanifold of  $V$ and  let $T$ be a closed positive current on $X$. Let $T_\infty$ be a tangent current to $T$ along $V$. Denote by $s$ the density h-dimension of $T_\infty$ and $[V']$. Then, the density h-dimension of $T$ and $[V']$  is at most $s$, and we have  
$$\kappa_s(T,[V']) \le \kappa_s(T_\infty, [V']).$$
The inequality still holds if we replace $s$ by the tangential h-dimension of $T$ along $V$.   
\end{theorem}

As a consequence, we obtain the following result generalizing the well-known lower bound of Lelong numbers of intersection of $(1,1)$-currents due to Demailly \cite[Page 169]{Demailly_analyticmethod} in the compact setting.

\begin{corollary}\label{cor-sosanhlelong} Let $T_j$ be a closed positive current on $X$ for $1 \le j \le m$.  Then, for every $x \in X$ and for every density current $S$ associated to $T_1, \ldots, T_{m}$, we have 
\begin{align}\label{inesosanhSTjlelong}
\nu(S, x) \ge \nu(T_1, x) \cdots \nu(T_m,x).
\end{align}
\end{corollary} 

\proof  Let $x \in X$. Let $\Delta_m$ be the diagonal of $X^m$.  Let   $\pi: E \to \Delta_m$ be the natural projection from the normal bundle of the diagonal $\Delta_m$ of $X^m$ in $X^m$. We identify $x$ with a point in $\Delta_m$ via the natural identification $\Delta_m$ with $X$.  Put $T:= \otimes_{j=1}^m T_j$ and $V':=\{x\}$. By \cite[Lemma 2.4]{Meo-auto-inter}, we have $\nu(T,x) \ge \nu(T_1,x) \cdots \nu(T_m,x).$   By \cite[Proposition 5.6]{Dinh_Sibony_density}, we have 
$$\kappa_0(S, [V'])= \nu(S,x), \quad \kappa_0(T,[V'])= \nu(T,x)$$
(notice here $\dim V' =0$).  This combined by Theorem \ref{th-sosanhVV'densityDS} applied to $X^m$, $T:= \otimes_{j=1}^m T_j$, $\Delta_m$ the diagonal of $X^m$ and $V':=\{x\}$  implies 
$$\nu(S,x)  \ge \nu(T,x) \ge \nu(T_1,x) \cdots \nu(T_m,x).$$
Hence, the desired inequality follows.   The proof is finished.  
\endproof

In the next part, we will use density currents to study the loss of mass of non-pluripolar products.  We need to recall basic properties of relative non-pluripolar products.  Let $T_1, \ldots, T_m$ be closed positive $(1,1)$-currents on $X$. By \cite{Viet-generalized-nonpluri}, the $T$-relative non-pluripolar product $\langle \wedge_{j=1}^m T_j \dot{\wedge} T\rangle$ is defined  in a way similar to that of  the usual non-pluripolar product. For readers' convenience, we explain briefly how to do it. 

Write $T_j= \ddc u_j+ \theta_j$, where $\theta_j$ is a smooth form and $u_j$ is a $\theta_j$-psh function. Put 
$$R_k:=\bold{1}_{\cap_{j=1}^m \{u_j >-k\}} \wedge_{j=1}^m (\ddc \max\{u_j,-k\} + \theta_j)\wedge T$$
for $k \in \N$.  By the strong quasi-continuity of bounded psh functions (\cite[Theorems 2.4 and 2.9]{Viet-generalized-nonpluri}), we have 
$$R_k= \bold{1}_{\cap_{j=1}^m \{u_j >-k\}} \wedge_{j=1}^m (\ddc \max\{u_j,-l\} + \theta_j)\wedge T$$
for every $l \ge k \ge 1$. Although it is not an immediate fact, one can check that  $R_k$ is positive (see \cite[Lemma 3.2]{Viet-generalized-nonpluri}). 

As in \cite{BEGZ}, we have that $R_k$ is of mass bounded uniformly in $k$ and $(R_k)_k$ converges to a closed positive current as $k \to \infty$. This limit is denoted by $\langle \wedge_{j=1}^m T_j \dot{\wedge} T\rangle$.  The  last product is, hence,  a well-defined closed positive current of bi-degree $(m+p,m+p)$; and it is  symmetric with respect to $T_1, \ldots, T_m$ and homogeneous. We refer to  \cite[Proposition 3.5]{Viet-generalized-nonpluri} for more properties of relative non-pluripolar products.

For every closed positive $(1,1)$-current $P$, we denote by $I_P$ the set of $x \in U$ so that local potentials of $P$ are equal to $-\infty$ at $x$. Note that $I_P$ is  a locally complete pluripolar set.  It is clear from the definition that $\langle \wedge_{j=1}^m T_j \dot{\wedge} T\rangle$ has no mass on $\cup_{j=1}^m I_{T_j}$. Furthermore, for every locally complete pluripolar set $A$ (\emph{i.e,} $A$ is locally equal to $\{\psi= -\infty\}$ for some psh function $\psi$), if $T$ has no mass on $A$, then so does $\langle \wedge_{j=1}^m T_j \dot{\wedge} T\rangle$.  The following is deduced from \cite[Proposition 3.5]{Viet-generalized-nonpluri}.

\begin{proposition}\label{pro-sublinearnonpluripolar} 
$(i)$ For $R:= \langle \wedge_{j=l+1}^m T_j \dot{\wedge} T \rangle$,  we have $\langle \wedge_{j=1}^m T_j \dot{\wedge} T \rangle = \langle \wedge_{j=1}^l T_j \dot{\wedge} R \rangle$.

$(ii)$ For every complete pluripolar set $A$, we have 
$$\bold{1}_{X \backslash A}\langle  T_1 \wedge T_2 \wedge \cdots \wedge T_m \dot{\wedge} T\rangle= \big\langle  T_1 \wedge T_2 \wedge \cdots \wedge T_m \dot{\wedge}(\bold{1}_{X \backslash A} T)\big\rangle.$$
In particular,  the equality
$$\langle \wedge_{j=1}^m T_j \dot{\wedge} T \rangle = \langle \wedge_{j=1}^m T_j \dot{\wedge} T' \rangle$$
holds, where $T':= \bold{1}_{X \backslash \cup_{j=1}^m I_{T_j}} T$.
\end{proposition}

Here is a crucial property of relative non-pluripolar products.

\begin{theorem} \label{th-monoticity} (\cite[Theorem 1.1]{Viet-generalized-nonpluri})  Let $T'_j$ be closed positive $(1,1)$-current in the cohomology class of $T_j$ on $X$ such that $T'_j$ is less singular than $T_j$ for $1 \le j \le m$. Then we have 
$$\{\langle T_1 \wedge \cdots \wedge T_m \dot{\wedge} T \rangle \} \le \{\langle T'_1 \wedge \cdots \wedge T'_m \dot{\wedge} T \rangle\}.$$
\end{theorem}

Weaker versions of the above result were proved in  \cite{BEGZ,Lu-Darvas-DiNezza-mono,WittNystrom-mono}.  Regarding the relation between relative non-pluripolar products and density currents, the following fact was proved in \cite[Theorem 3.5]{Viet-density-nonpluripolar}, see also \cite{VietTuanLucas,Viet_Lucas}.

\begin{theorem} \label{the-phanbucuarestricteddenstyvanonpluri}  Let $R_\infty$ be a density current associated to $T_1, \ldots, T_m,T$. Then we have 
\begin{align} \label{ine_TjS2}
\pi^* \langle \wedge_{j=1}^m T_j  \dot{\wedge} T \rangle \le R_\infty,
\end{align} 
where $\pi$ is the natural projection from the normal bundle of the diagonal $\Delta$ of $X^{m+1}$ to $\Delta$, and as usual we identified $\Delta$ with $X$. 
\end{theorem}

We will need the following to estimate the density h-dimension of currents. 

\begin{proposition} \label{pro-uocluongmassonsmallsetdensity} (\cite[Proposition 3.6]{Viet-density-nonpluripolar}) Let $A$ be a Borel subset of $X$. Assume that for every density current $R_J$ associated to $(T_j)_{j \in J},T$ with $m_J:=|J|<m$, we have that $R_J$ has  no mass on the set $$\pi_{m_J+1}^*(\cap_{j \not \in J} \{x \in A: \nu(T_j, x) >0\}).$$
 Then for every density current $S$ associated to $T_1, \ldots, T_m,T$, the h-dimension of the current $\bold{1}_A S$ is equal to $n-p-m$. 
\end{proposition}

Let $P$ and $T$ be closed positive current of bi-degree $(1,1)$ and  $(p,p)$ respectively on $X$, where $1 \le p \le n-1$. 

\begin{lemma} \label{le-truonghom=1lelong} Assume that $T$ has no mass on $I_{P}$.  Then, the cohomology class 
$$\gamma:= \{P\} \wedge \{T\}- \{\langle P \dot{\wedge} T\rangle \}$$
 is pseudoeffective and  we have 
\begin{align}\label{ine-uocluonggammaVPT}
\| \gamma\| \ge \sum_{V} \nu(P,V)\nu(T,V) \vol(V),
\end{align}
where the sum is taken over every irreducible subset $V$ of dimension at least $n-p-1$ in $X$.   
\end{lemma}

Here $\nu(P,V)$ and $\nu(T,V)$ denote the generic Lelong numbers of $P$ and $T$ along $V$, respectively.

\proof  Let $\cali{V}$ be the set of irreducible analytic subsets $V$ of dimension at least $n-p-1$ in $X$ such that $\nu(T,V)>0$ and $\nu(P,V)>0$. We note that in (\ref{ine-uocluonggammaVPT}), it is enough to consider  $V \in \cali{V}$. We will see below that $\cali{V}$ has at most countable elements.     

Observe that if $\nu(P,x)>0$, then $x \in I_P$. Hence, by hypothesis, the trace measure of $T$ has no mass on the set $\{x \in X: \nu(P,x)>0\}$. This allows us to apply Proposition \ref{pro-uocluongmassonsmallsetdensity} to $P$ and $T$ to obtain that the density h-dimension of $P$ and $T$ is minimal. Using this and Lemma \ref{le-classDSproduct} gives 
\begin{align}\label{eq-densityclassPT}
\kappa(P,T)= \pi^* (\{P\} \wedge \{T\}),
\end{align}
where $\pi$ is the natural projection from the normal bundle of the diagonal $\Delta$ of $X^2$ to $\Delta$.  

Let $S$ be a density current associated to $P$ and $T$. Since the h-dimension of $S$ is minimal, using Lemma \ref{le-minimalhdimension}, we get that there exists a current $S'$ on $X$ such that $S= \pi^* S'$ (recall $\Delta$ is identified with $X$). Since the relative non-pluripolar product is dominated by density currents (Theorem \ref{the-phanbucuarestricteddenstyvanonpluri}), the current $S' - \langle P \dot{\wedge} T \rangle$ is closed and positive. Moreover, by (\ref{eq-densityclassPT}), the cohomology class of the last current is equal to $\gamma$. It follows that $\gamma$ is pseudoeffective. 

It remains to prove (\ref{ine-uocluonggammaVPT}).  Let $V \in \cali{V}$. By definition, the generic Lelong number of $T$ along $V$ is positive. Since $T$ is of bi-degree $(p,p)$, the dimension of $V$ must be at most $n-p$. Hence, we have two possibilities: either $\dim V= n-p-1$ or $\dim V= n-p$. The latter case cannot happen because $T$ would have mass on $V$ which is contained in $I_P$ (for $\nu(P,V)>0$), this contradicts the hypothesis that $T$ has no mass on $I_P$.  Hence, for every $V \in \cali{V}$, we have $\dim V= n-p-1$. We also deduce that for $V,V' \in \cali{V}$, then $V \not \subset V'$. 
  
Applying Corollary \ref{cor-sosanhlelong} to $P,T$ and generic $x \in V$ gives 
$$\nu(S,V) \ge \nu(P,V) \nu(T,V).$$
This combined with the fact that $\dim V= n-p-1$ implies  $S \ge \nu(P,V) \nu(T,V) \, [V]$.
We deduce that 
$$S \ge \sum_{V \in\cali{V}}\nu(P,V) \nu(T,V) \, [V].$$
The desired assertion follows. The proof is finished.
\endproof

\section{Reversed Alexandrov-Fenchel inequalities} \label{sec-AF}

We first recall an integration by parts formula for relative non-pluripolar products from \cite{Viet-convexity-weightedclass} generalizing those given in \cite{BEGZ,Lu-cap-compare,Xia}.     

Let $X$ be a compact K\"ahler manifold. Recall that a \emph{dsh} function on $X$ is the difference of two quasi-plurisubharmonic (quasi-psh for short) functions on $X$ (see \cite{DS_tm}). These functions are well-defined outside pluripolar sets. Let $v$ be a dsh function on $X$. The last function is said to be \emph{bounded} in $X$ if there exists a constant $C$ such that $|v| \le C$ on $X$ (outside certain pluripolar set).   

Let $T$ be a closed positive current on $X$. We say that $v$ is \emph{$T$-admissible} if  there exist  quasi-psh functions $\varphi_1, \varphi_2$ on $X$ such that $v= \varphi_1- \varphi_2$  and $T$ has no mass on $\{\varphi_j=-\infty\}$ 
for $j=1,2$. In particular, if $T$ has no mass on pluripolar sets, then every dsh function is $T$-admissible.  

Assume now that $v$ is  \emph{bounded} $T$-admissible.    Let $\varphi_{1}, \varphi_{2}$ be quasi-psh functions such that $v= \varphi_{1}- \varphi_{2}$ and $T$ has no mass on $\{\varphi_{j}=-\infty\}$ for $j=1,2$. Let 
$$\varphi_{j,k}:= \max\{\varphi_{j}, -k \}$$
for every $j=1,2$ and $k \in \N$. Put $v_k:= \varphi_{1,k}- \varphi_{2,k}$ and
$$Q_k:= d v_k \wedge \dc v_k \wedge T=\ddc v_k^2  \wedge T - v_k\ddc v_k \wedge T.$$
By the plurifine locality with respect to $T$ (\cite[Theorem 2.9]{Viet-generalized-nonpluri}) applied to the right-hand side of the last equality, we have 
\begin{align}\label{eq-localplurifineddc}
\bold{1}_{\cap_{j=1}^2 \{\varphi_{j}> -k\}} Q_k =\bold{1}_{\cap_{j=1}^2\{\varphi_{j}> -k\}} Q_{k'}
\end{align}
for every $k'\ge k$. By \cite[Lemma 2.5]{Viet-convexity-weightedclass},  the mass of $Q_k$ on $X$ is bounded uniformly in $k$.  This combined with  (\ref{eq-localplurifineddc}) implies that there exists a positive current $Q$ on $X$ such that for every bounded Borel form $\Phi$ with compact support on $X$ such that 
$$\langle Q,\Phi \rangle  = \lim_{k\to \infty} \langle Q_k, \Phi\rangle.$$
We define $\langle d v \wedge \dc v \dot{\wedge} T \rangle$ to be the current $Q$.  This agrees with the classical definition if $v$ is the difference of two  bounded quasi-psh functions. One can check  that this definition is independent of the choice of $\varphi_1, \varphi_2$.

Let $w$ be another bounded $T$-admissible dsh function.  If $T$ is of bi-degree $(n-1,n-1)$, we can also define the current $\langle  dv \wedge \dc w \dot{\wedge} T\rangle$ by a similar procedure as above.  We put
$$\langle \ddc v \dot{\wedge} T \rangle:= \langle \ddc \varphi_1 \dot{\wedge} T\rangle- \langle \ddc \varphi_2 \dot{\wedge} T\rangle$$
which is independent of the choice of $\varphi_1,\varphi_2$. By $T$-admissibility, we have 
$$\langle  \ddc (v+w)  \dot{\wedge} T\rangle=\langle  \ddc v\dot{\wedge} T\rangle+\langle  \ddc w   \dot{\wedge} T\rangle.$$ 
Here is an integration by parts formula for relative non-pluripolar products.

\begin{theorem} \label{th-integra} (\cite[Theorem 2.6]{Viet-convexity-weightedclass})  Let  $T$ a closed positive current of bi-degree $(n-1,n-1)$ on $X$. Let $v$ and $w$ be bounded $T$-admissible dsh functions on $X$. Then,  we have  
\begin{align}\label{eq-intebyparts}
\int_X w \langle  \ddc v \dot{\wedge} T\rangle=\int_X v \langle \ddc w\dot{\wedge} T\rangle=- \int_X \langle dw \wedge \dc v \dot{\wedge} T \rangle.
\end{align}
\end{theorem}


Let $\theta$ be  a closed smooth $(1,1)$-form on $X$. Given a $\theta$-psh function $u$, we use the usual notation that 
$$\theta_u:= \ddc u+ \theta.$$
Let $T$ be a closed positive current of bi-degree $(n-2,n-2)$ on $X$. Let $u, v, \varphi$ be $\theta$-psh functions such that $u\le \varphi$ and $v \le \varphi$.  

The following can be regarded as an inequality of reversed Alexandrov-Fenchel type. Several related estimates for mixed Monge-Amp\`ere operators were obtained in the local setting; see \cite{Demailly-Hiep,Kim-Rashkovskii}.

\begin{proposition}\label{pro-theoremmain} Assume that 
\begin{align}\label{eq-hypo-fullmasthetau}
\int_X \langle \theta_\varphi^2\dot{\wedge} T\rangle = \int_X \langle \theta_u \wedge \theta_\varphi \dot{\wedge} T\rangle,
\end{align}
and $T$ has no mass on $\{u= -\infty\}\cup\{v= - \infty\}$. Then, we have 
\begin{align}\label{ine-reversedAF}
\int_X \big(\langle \theta_\varphi \wedge \theta_v \dot{\wedge} T\rangle - \langle \theta_u \wedge \theta_v \dot{\wedge} T\rangle\big)  \le \bigg( \int_X \big(\langle \theta_\varphi^2\dot{\wedge} T\rangle -  \langle \theta_u^2 \dot{\wedge} T\rangle\big) \bigg)^{1/2} \bigg(  \int_X \big(\langle \theta_\varphi^2 \dot{\wedge} T\rangle -  \langle \theta_v^2\dot{\wedge} T\rangle\big)  \bigg)^{1/2}
\end{align}
\end{proposition}

Notice that by the monotonicity of relative non-pluripolar products, the left-hand side of (\ref{ine-reversedAF}) is non-negative.

\proof    By hypothesis, one has
\begin{align}\label{eq-massofTonuinftry}
\bold{1}_{\{u= -\infty\} \cup \{v= -\infty\}}T=0.
\end{align}
Note here that 
$$\{\varphi = -\infty\} \subset \{u= -\infty\}\cap \{v= -\infty\}.$$
 Let 
$$u_{k}:= \max\{u, \varphi-k\}- (\varphi -k)$$
and
 $$\psi_k:= k^{-1}\max\{u+ v-  2\varphi, - k \}+1=k^{-1}\max\{u+ v,  2\varphi - k \}- k^{-1} 2 \varphi +1$$
Define $v_k$ similarly. We have $\psi_k= 0$ on $ \{u \le \varphi -k\} \cup \{v \le \varphi -k\}$.  Note also that $0 \le u_{k}, v_k \le k$ for $j=1,2$ (hence $u_k,v_k,\psi_k$ are bounded $T$-admissible dsh functions by (\ref{eq-massofTonuinftry})) and 
$$\ddc \psi_k + k^{-1}\eta\ge 0,$$
where $\eta:=  2 \theta_{\varphi}$.   We can check that 
\begin{align*}
\langle \theta_{u} \wedge \theta_{v}  \dot{\wedge} T \rangle& =\lim_{k \to \infty} \psi_k \big\langle (\ddc u_{k}+\theta_{\varphi})\wedge (\ddc v_{k}+\theta_{\varphi}) \dot{\wedge} T \big\rangle.
\end{align*}
Put 
$$B_k:= \int_X \psi_k \big\langle \theta_{\varphi}\wedge (\ddc v_{k}+\theta_{\varphi}) \dot{\wedge} T \big\rangle - \int_X \psi_k \big\langle (\ddc u_{k}+\theta_{\varphi})\wedge (\ddc v_{k}+\theta_{\varphi}) \dot{\wedge} T \big\rangle$$
and
$$A:= \int_X \big(\langle \theta_\varphi \wedge \theta_v \dot{\wedge} T\rangle - \langle \theta_u \wedge \theta_v \dot{\wedge} T\rangle\big)$$
By (\ref{eq-massofTonuinftry}) and Proposition \ref{pro-sublinearnonpluripolar} $(ii)$, we have
$$A= \lim_{k \to \infty} B_k.$$
Observe
\begin{align*}
B_k &=  - \int_X \psi_k \langle \ddc u_k  \wedge (\ddc v_{k}+\theta_{\varphi}) \dot{\wedge} T \big\rangle\\
 &=   - \int_X \psi_k \langle \ddc u_k  \wedge \ddc v_{k} \dot{\wedge} T \big\rangle   - \int_X \psi_k \langle \ddc u_k  \wedge \theta_{\varphi} \dot{\wedge} T \big\rangle.
\end{align*}
Denote by $I_1, I_2$ the first and  second term in the right-hand side of the last equality. Using (\ref{eq-massofTonuinftry}) gives
\begin{align*}
\lim_{k \to \infty} I_2 &= \lim_{k \to \infty} \bigg(- \int_X \psi_k \langle (\ddc u_k+ \theta_\varphi)  \wedge \theta_{\varphi} \dot{\wedge} T \big\rangle+ \int_X \psi_k \langle \theta_{\varphi}^2 \dot{\wedge} T \big\rangle\bigg)\\
&=  -  \int_X\langle   \theta_{u}\wedge \theta_{\varphi} \dot{\wedge} T \big\rangle+\int_X \langle \theta_{\varphi}^2 \dot{\wedge} T \big\rangle=0
 \end{align*}
 by (\ref{eq-hypo-fullmasthetau}). 
 Thus we get 
 \begin{align}\label{ine-buocdautienABD}
B_k  =I_1+ o_{k \to \infty}(1).
\end{align}
 Theorem \ref{th-integra} applied to the formula defining $I_1$ gives 
\begin{align}\label{eq-biendoiI1}
- I_1  &= \int_X  u_{k} \big\langle   \ddc \psi_k \wedge \ddc v_{k} \dot{\wedge} T \big\rangle\\
\nonumber
&=  -\int_X \big \langle d u_{k} \wedge \dc v_{k} \wedge   \ddc \psi_k  \dot{\wedge} T \big\rangle\\
\nonumber
&= -\int_X  \big \langle d u_k \wedge \dc v_k \wedge   (\ddc \psi_k+k^{-1}\eta)  \dot{\wedge} T \big\rangle + k^{-1}\int_X \big \langle d u_k \wedge \dc v_k \wedge   \eta  \dot{\wedge} T \big\rangle.
\end{align}
Denote by $J_1,J_2$ the first and second term in the right-hand side of the last equality. We treat $J_2$. By integration by parts (Theorem \ref{th-integra}), we obtain
\begin{align} \label{eq-biendoiJ2}
J_2 &= - k^{-1} \int_X   u_k \langle \ddc v_k \wedge   \eta  \dot{\wedge} T \rangle
&=  - k^{-1} \int_X   u_k \langle (\ddc v_k+\theta_{\varphi}) \wedge   \eta  \dot{\wedge} T \rangle+ o_{k \to \infty}(1)
\end{align}
because of (\ref{eq-massofTonuinftry}) and the fact that $u_k/k$ converges to $0$ on $\{u>- \infty\}$ and to  $-1$ otherwise.  Using this and noticing that $u_k\le 0$,  we infer
\begin{align} \label{ine-uocluongJ2}
\liminf_{k\to \infty} J_2 \ge  0.
\end{align}
On the other hand, using the Cauchy-Schwarz inequality, we obtain
\begin{align}\label{ine-uocliongJ1}
J_1^2  \le  \int_X \big \langle d u_k \wedge \dc u_k \wedge   (\ddc \psi_k+k^{-1}\eta)  \dot{\wedge} T \big\rangle   \int_X\big \langle d v_k \wedge \dc v_k \wedge   (\ddc \psi_k+k^{-1}\eta)  \dot{\wedge} T \big\rangle.
\end{align}
Denote by $J_{11}, J_{12}$ the first and second term in the right-hand side of the last inequality. Put 
$$J'_{11}:=k^{-1} \int_X\big \langle d u_k \wedge \dc u_k \wedge  \eta  \dot{\wedge} T \big\rangle.$$
Using integration by parts and arguing as in (\ref{eq-biendoiJ2}), we obtain 
\begin{align} \label{ine-uocluongJphay11sua}
J'_{11} &=-k^{-1} \int_X  u_k \big \langle \ddc u_k \wedge  \eta  \dot{\wedge} T \big\rangle\\
\nonumber
&= - k^{-1} \int_X  u_k \big \langle (\ddc u_k+ \theta_\varphi) \wedge  \eta  \dot{\wedge} T \big\rangle+ o_{k \to\infty}(1)\\
\nonumber
&= \int_X \big \langle (\ddc u_k+ \theta_\varphi) \wedge  \eta  \dot{\wedge} T \big\rangle - \int_X (u_k/k +1)  \big \langle (\ddc u_k+ \theta_\varphi) \wedge  \eta  \dot{\wedge} T \big\rangle+ o_{k \to\infty}(1) \\
\nonumber
&= \int_X \big \langle \theta_\varphi \wedge  \eta  \dot{\wedge} T \big\rangle - \int_X (u_k/k +1)  \big \langle \theta_u \wedge  \eta  \dot{\wedge} T \big\rangle+ o_{k \to\infty}(1) 
\end{align}
because $u_k/k+1=0$ on $\{u \le \varphi -k\}$. Letting  $k \to \infty$ in (\ref{ine-uocluongJphay11sua}) gives 
\begin{align}\label{ine-uocluongJphay11}
\lim_{k\to \infty} J'_{11}= \int_X   \big \langle \theta_{\varphi} \wedge   \eta  \dot{\wedge} T \big\rangle-\int_X   \big \langle \theta_{u} \wedge   \eta  \dot{\wedge} T \big\rangle=2 \int_X   \big \langle \theta_{\varphi}^2  \dot{\wedge} T \big\rangle- 2\int_X   \big \langle \theta_{u} \wedge   \theta_\varphi  \dot{\wedge} T \big\rangle.
\end{align}
Using Theorem \ref{th-integra} again and arguing as in (\ref{ine-uocluongJphay11sua}), we have 
\begin{align*}
J_{11} &=\int_X \big \langle d u_k \wedge \dc u_k \wedge  \ddc \psi_k \dot{\wedge} T \big\rangle+ J'_{11}\\
&= -\int_X \psi_k \big \langle \ddc u_k \wedge \ddc u_k  \dot{\wedge} T \big\rangle+ J'_{11}\\
&= - \int_X \psi_k \big \langle  (\ddc u_k+\theta_\varphi)^2  \dot{\wedge} T \big\rangle + 2 \int_X \psi_k \big \langle \ddc u_k \wedge \theta_{\varphi} \dot{\wedge} T \big\rangle + \int_X \psi_k \big \langle \theta_{\varphi}^2 \dot{\wedge} T \big\rangle   + J'_{11}\\
&= \int_X \big\langle \theta_\varphi^2  \dot{\wedge} T \big\rangle -\int_X \big\langle \theta_u^2  \dot{\wedge} T \big\rangle+ 2 \int_X \big\langle \theta_u \wedge \theta_\varphi  \dot{\wedge} T \big\rangle- 2\int_X \big\langle \theta_\varphi^2  \dot{\wedge} T \big\rangle + o_{k \to\infty}(1)+ J'_{11}
\end{align*}
This combined with (\ref{ine-uocluongJphay11}) yields that 
$$\limsup_{k \to \infty} J_{11} \le  \int_X \big \langle \theta^2_{\varphi} \dot{\wedge} T \big\rangle- \int_X \big \langle \theta_{u}^2 \dot{\wedge} T \big\rangle.$$
By similar computations, we also get
$$\limsup_{k \to \infty} J_{12} \le  \int_X \big \langle \theta^2_{\varphi} \dot{\wedge} T \big\rangle- \int_X \big \langle \theta_{v}^2 \dot{\wedge} T \big\rangle.$$
 This together with (\ref{eq-biendoiI1}) and (\ref{ine-buocdautienABD}) gives 
$$A=\limsup_{k \to \infty} B_k=\limsup_{k \to \infty} I_1= -  \liminf_{k \to \infty} -I_1 \le C,$$
 where $C$ is the right-hand side of (\ref{ine-reversedAF}). This finishes the proof.
\endproof

\begin{corollary}\label{cor-theoremmain} Let $T$ be closed positive current of bi-dimension $(m,m)$ on $X$ and let  $\theta$ be  a closed smooth $(1,1)$-form on $X$. Let $u, v, \varphi$ be $\theta$-psh functions such that $u,v \le \varphi$, and  $T$ has no mass on $\{u= -\infty\}$ and $\{v= - \infty\}$. Let $M$ be a positive constant greater than $\int_X \theta_\varphi^m \dot{\wedge} T$.  Then, there exists a constant $C_M$ depending only on $M,m,n$ such that 
\begin{align}\label{ine-thetauvtachldinhluong}
\bigg| \int_X \langle \theta_\varphi^{m-l} \wedge \theta_v^l  \dot{\wedge} T\rangle - \int_X \langle \theta_u^{m-l} \wedge \theta_v^l \dot{\wedge} T\rangle \bigg | \le C_M \bigg(\int_X \langle \theta_\varphi^m\dot{\wedge} T\rangle - \int_X \langle \theta_u^m \dot{\wedge} T\rangle\bigg)^{2^{-l}}.
\end{align}
In particular, if
$$\int_X \langle \theta_\varphi^m\dot{\wedge} T\rangle = \int_X \langle \theta_u^m \dot{\wedge} T\rangle,$$
then, for every integer $0 \le l \le m$, we have 
$$\int_X \langle \theta_\varphi^{m-l} \wedge \theta_v^l  \dot{\wedge} T\rangle=  \int_X \langle \theta_u^{m-l} \wedge \theta_v^l \dot{\wedge} T\rangle.$$
\end{corollary}

\proof In what follows, we denote by $C_M$ a positive constant depending only on $M$ and $n$, and the value of $C_M$ might vary from line to line. Put 
$$I:= \int_X \langle \theta_\varphi^m\dot{\wedge} T\rangle - \int_X \langle \theta_u^m \dot{\wedge} T\rangle.$$
 By the monotonicity of relative non-pluripolar products, we get 
\begin{align}\label{eq-hypo-fullmasthetaucor}
0 \le \int_X \langle \theta_\varphi^m\dot{\wedge} T\rangle - \int_X \langle \theta_\varphi^{m-l} \wedge \theta_u^{l} \dot{\wedge} T\rangle \le I
\end{align} 
for every $0 \le l \le m$. Observe that  the desired assertion in the case where $m=2$ is a direct consequence of Proposition \ref{pro-theoremmain}.  We prove  by induction on $l'$ that for every $0\le l_1, l_2\le m$ with $l_1+l_2 \le l'$ we have 
\begin{align}\label{eq-induclwedgemthetavarphi}
\int_X \langle \theta_\varphi^{m-l_2} \wedge \theta_v^{l_2} \dot{\wedge} T\rangle- \int_X \langle \theta_\varphi^{m-l_1-l_2} \wedge \theta_u^{l_1} \wedge \theta_v^{l_2}  \dot{\wedge} T\rangle \le C_M \, I^{2^{-l_2}},
\end{align}
for some constant $C_M$ depending only on $M,m$ and $n$. 

The desired inequality is a special case of (\ref{eq-induclwedgemthetavarphi}) when $l_1= m-l_2$.  When $l'=0$, the inequality (\ref{eq-induclwedgemthetavarphi}) is clear. Suppose that  (\ref{eq-induclwedgemthetavarphi}) holds for every $l_1,l_2$ with  $l_1+ l_2 \le l'-1$. We need to prove it for $l'$ in place of $l'-1$.  To this end, we now use another induction on $0 \le l_2 \le l'$ to prove the statement $(*)$ that (\ref{eq-induclwedgemthetavarphi}) holds for every $l_1$ with $l_1+ l_2 \le l'$. When $l_2=0$, the statement $(*)$ is a direct consequence of (\ref{eq-hypo-fullmasthetaucor}). Assume now that $(*)$ holds for $l_2-1$. We now prove it for $l_2$.    Let 
$$T':=\langle \theta_\varphi^{m-l_1-l_2} \wedge \theta_u^{l_1-1} \wedge \theta_v^{l_2-1} \dot{\wedge} T\rangle.$$
 We have
 $$  \langle \theta_\varphi^{m-l_1-l_2} \wedge \theta_u^{l_1} \wedge \theta_v^{l_2} \dot{\wedge} T\rangle= \langle \theta_u \wedge \theta_v \dot{\wedge} T'\rangle,$$
 and 
\begin{align*}
\int_X \langle \theta_u^2 \dot{\wedge} T'\rangle =\int_X \langle \theta_\varphi^{m-l_1-l_2} \wedge \theta_u^{l_1+1} \wedge \theta_v^{l_2-1} \dot{\wedge} T\rangle
\end{align*}
and
\begin{align*}
\int_X \langle \theta_\varphi^2 \dot{\wedge} T'\rangle= \int_X \langle \theta_\varphi^{m-l_1-l_2+2}  \wedge \theta_u^{l_1-1} \wedge \theta_v^{l_2-1} \dot{\wedge} T\rangle.
\end{align*}
By this and the induction hypothesis on $l_2$ that $(*)$ holds for $l_2-1$ and the monotonicity of relative non-pluripolar products, we obtain that
\begin{align*}
\int_X \langle \theta_\varphi^2 \dot{\wedge} T'\rangle- \int_X \langle \theta_u^2 \dot{\wedge} T'\rangle &\le \int_X \langle \theta_\varphi^{m-l_2+1} \wedge \theta_v^{l_2-1} \dot{\wedge} T\rangle-   \int_X \langle \theta_\varphi^{m-l_1-l_2} \wedge \theta_u^{l_1+1} \wedge \theta_v^{l_2-1} \dot{\wedge} T\rangle\\
&\le C_M I^{2^{-l_2+1}}.
\end{align*}
Note that $T'$ has no mass on $\{u= -\infty\}$ and $\{v= - \infty\}$ because $T$ does so; see \cite[Lemma 2.1]{Viet-generalized-nonpluri}.  Applying Proposition \ref{cor-theoremmain} to $u,v,\varphi, T'$ gives
$$  \int_X\langle \theta_\varphi \wedge \theta_v \dot{\wedge} T'\rangle-  \int_X\langle \theta_u \wedge \theta_v \dot{\wedge} T'\rangle \le C_M \bigg( \int_X \langle \theta_\varphi^2 \dot{\wedge} T'\rangle- \int_X \langle \theta_u^2 \dot{\wedge} T'\rangle\bigg)^{1/2} \le C_M \, I^{2^{-l_2}}.$$
We deduce that 
$$ \int_X \langle \theta_\varphi^{m-l_1-l_2+1} \wedge \theta_u^{l_1-1} \wedge \theta_v^{l_2} \dot{\wedge} T\rangle-   \int_X\langle \theta_\varphi^{m-l_1-l_2} \wedge \theta_u^{l_1} \wedge \theta_v^{l_2} \dot{\wedge} T\rangle \le C_M \, I^{2^{-l_2}}.$$
On the other hand, using the induction hypothesis on $l'-1$ gives
$$  \int_X\langle \theta_\varphi^{m-l_2} \wedge \theta_v^{l_2} \dot{\wedge} T\rangle - \int_X \langle \theta_\varphi^{m-l_1-l_2+1} \wedge \theta_u^{l_1-1} \wedge \theta_v^{l_2} \dot{\wedge} T\rangle \le C_M \, I^{2^{-l_2}}.$$
Summing up the last two inequalities gives the desired assertion $(*)$ for $l_2$. In other words, (\ref{eq-induclwedgemthetavarphi}) holds for every $l_1,l_2$ with $l_1+ l_2 \le l'$. This is what we want to prove. The proof is finished.
\endproof

We say that a  closed positive $(1,1)$-current $P$ is \emph{a current with analytic singularities associated to a coherent analytic ideal sheaf $\cali{S}$ on $X$} if locally on $X$, every local potential $u$ of $P$ satisfies
$$u=c \log \sum_{j=1}^M |f_j|+O(1),$$
where $c> 0$ is a constant, and  $\{f_1,\ldots, f_M\}$ are local generators of $\cali{S}$; see \cite[Definition 9.1]{Boucksom_l2}. By Demailly's analytic approximation of psh functions, every big cohomology class contains such a current $P$ which is K\"ahler.

We fix a smooth K\"ahler form $\omega$ on $X$ with $\int_X \omega^n =1$. For every positive $(p,p)$-current $S$, the mass of $S$ is defined by 
$$\|S\|:= \int_X S \wedge \omega^{n-p}.$$
Similarly, we can define the mass norm of a pseudoeffective cohomology classes.  

\begin{theorem} \label{the-congthemomega}  Let $1 \le m \le n$ be an integer. Let $\alpha$ be a big cohomology $(1,1)$-class in $X$. Let $T$ and $T'$ be closed positive currents in $\alpha$ such that $T'$ is more singular than $T$. Assume that there exists a closed positive current $P$ with analytic singularities associated to a coherent ideal sheaf such that $P$ is more singular than $T$.  Let $\theta$ be a closed continuous real $(1,1)$-form. Let $M>0$ be constant greater than $\|\alpha\|$, $\|\langle \alpha^m \rangle\|$ and $\|\theta\|_{\cali{C}^0}$ such that $P \ge M^{-1} \omega$.   Then, there exists a constant $C_M>0$ depending only on $M,m,\omega,X$ such that
\begin{align}\label{eq-TTcongvoiomega}
 \big \| \{\langle (T'+\theta)^m \rangle\} -\{\langle (T+\theta)^m \rangle\} \big\| \le  C_M\big\|  \{\langle T'^m \rangle\}- \{\langle T^m \rangle\}\big\|^{2^{-m}}.
\end{align}
In particular, if $T+\theta$ and $T'+ \theta$ are positive, then there holds
\begin{align}\label{state-tuonduongTcongtheta}
\{\langle T'^m \rangle\}= \{\langle T^m \rangle\}  \quad \text{if and only if}\quad  \{\langle (T'+\theta)^m \rangle\}= \{\langle (T+\theta)^m \rangle\}.
\end{align}
\end{theorem}

Examples of $T$ in the above result are currents with minimal singularities in $\alpha$. The assertion (\ref{state-tuonduongTcongtheta})  was known when $\alpha$ is K\"ahler (see \cite{DiNezza-stability,Viet-generalized-nonpluri}) and seems to be new for big classes even when $m=n$. 

\begin{proof} As above we denote by $C_M$ a positive constant depending only on $M,m,\omega,X$ whose value can vary from line to line.  We can choose such a  $C_M$ big enough such that $\omega':=\theta+ C_M \omega \ge 0$.  By the multi-linearity of non-pluripolar products, it is sufficient to prove (\ref{eq-TTcongvoiomega}) for $\omega'$ in place of $\theta$. In order to do so, we only need to check that 
\begin{align}\label{ine-lmhieuTTphay}
 \big \| \{\langle T'^l \rangle\}- \{\langle T^l \rangle \}\big \| \le C_M \big \| \{\langle T'^m \rangle\}- \{\langle T^m \rangle \}\big \|^{2^{-m}}
\end{align}
 for every $1 \le l \le m$. 

Principalizing the ideal sheaf associated to $P$ (see \cite{Hironaka-desingular,Wlodarczyk}), we get a smooth modification $\sigma: \widehat X \to X$ such that $\widehat P:= \sigma^*P= P_1+ P_2$, where $P_1$ is  a linear combination with nonnegative coefficients of currents of integration along smooth hypersurfaces, and $P_2$ is a closed positive smooth form. 
Since $P \ge M^{-1} \omega$, we get 
\begin{align}\label{ine-P2omefa}
P_2 \ge M^{-1} \sigma^* \omega
\end{align}
Let $\widehat T:= \sigma^* T$ and $\widehat T':= \widehat T'$. Since non-pluripolar products have no mass on analytic sets, using the hypothesis, we have
$$\int_{\widehat X}\langle \widehat T^m \rangle \wedge \sigma^* \omega^{n-m}=\int_{X}\langle T^m \rangle \wedge \omega^{n-m}, \quad \int_{X}\langle T'^m \rangle \wedge \omega^{n-m}=\int_{\widehat X}\langle \widehat T'^m \rangle \wedge \sigma^* \omega^{n-m}.$$
Using this and applying Corollary \ref{cor-theoremmain} to potentials of $\widehat T, \widehat T', \widehat P$ give
$$\int_{\widehat X} \big(\langle \widehat T^l \wedge \widehat P^{m-l} \rangle- \langle \widehat T'^l \wedge \widehat P^{m-l} \rangle\big) \wedge \sigma^* \omega^{n-m} \le C_M \| \langle T^m \rangle - \langle T'^m \rangle\|^{2^{-m+l}}$$
for every $1 \le l \le m$. 
Let $R$ be a closed positive current representing the cohomology class $\{\langle \widehat T^l \rangle\}-\{\langle \widehat T'^l \rangle\}$. By the above inequality  and the smoothness of $P_2$, we obtain 
$$\int_{\widehat X} R \wedge P_2^{m-l} \wedge \sigma^* \omega^{n-m} \le C_M \| \langle T^m \rangle - \langle T'^m \rangle\|^{2^{-m+l}}.$$
This combined with (\ref{ine-P2omefa}) gives
$$\int_{\widehat X} R \wedge \sigma^* \omega^{n-l} \le C_M \| \langle T^m \rangle - \langle T'^m \rangle\|^{2^{-m+l}}.$$ 
It follows that 
\begin{align*}
\int_X \big(\langle T^l \rangle - \langle T'^l \rangle \big) \wedge \omega^{n-l} &=\int_{\widehat X} \big(\langle \widehat T^l \rangle- \langle \widehat T'^l \rangle) \wedge \sigma^* \omega^{n-l}\\
&=\int_{\widehat X} R \wedge \sigma^* \omega^{n-l} \le C_M \| \langle T^m \rangle - \langle T'^m \rangle\|^{2^{-m+l}}.
\end{align*}
The desired assertion follows because $\|\langle T^m \rangle\| \le \|\langle \alpha^m \rangle\|$. 

We now prove (\ref{state-tuonduongTcongtheta}). The implication $\Rightarrow$ is clear. The other one is obtained from the first one by applying to $T+\theta, T'+\theta$ in place of $T, T'$ respectively, and $-\theta$ in place of $\theta$.  The proof is complete.
\end{proof}

\begin{corollary} \label{cor-Tcongvoitheta} Let $1 \le m \le n$ be an integer. Let $\alpha$ be a big cohomology $(1,1)$-class in $X$. Let $T$ be  a closed positive current in $\alpha$ such that $\{\langle T^m \rangle\}= \langle \alpha^m \rangle$. Then we have  
$$\{\langle T^l \rangle\}= \langle \alpha^l \rangle$$
for every $1\le l \le m$. In particular, if $\alpha$ is big nef and $\theta$ is a closed continuous real $(1,1)$-form representing a cohomology class $\gamma$, then  
\begin{align}\label{eq-congthetabignef}
\{\langle (T+ \theta)^m\rangle \} = (\alpha +\gamma)^m.
\end{align}
\end{corollary}

The equality (\ref{eq-congthetabignef}) was proved in  \cite[Theorem 1.1]{Lu-Darvas-DiNezza-singularitytype} in the case where $m=n$. Note that in the last case, if additionally $T+\theta \ge 0$, then using \cite[Corollary 6.4]{Demailly_regula_11current} and the fact that the Lelong numbers of $T$ vanish, we get that $\alpha+\gamma$ is nef.

\begin{proof}  
Let $T_{\min}$ be a current with minimal singularity in $\alpha$. By Theorem \ref{the-congthemomega} for $T$ and $T_{\min}$ and the hypothesis, we get
$$\{\langle T+\omega)^m\rangle\}= \{\langle (T_{\min}+ \omega)^m\rangle\}.$$
Expanding the sums in both sides and using the fact that $\{\langle T^l\rangle\}\le \langle T_{\min}^l \rangle$ for every $1\le l \le m$, we get the first  desired equality.
The second desired equality follows immediately from the first one. This finishes the proof. 
\end{proof}

\begin{corollary} \label{cor-Tcongvoitheta2} Let $\alpha_1$ and $\alpha_2$ be big nef cohomology classes  on $X$. Let $T_j$ be a closed positive current in $\alpha_j$ for $j=1,2$. Then, 
\begin{align}\label{eq-congu1u2theta}
\big\{\big\langle (T_1+T_2)^m\big\rangle \big\} = \langle (\alpha_1+\alpha_2)^m \rangle
\end{align}
 if and only if   
\begin{align}\label{eq-congu1u2thetatach}
\{\langle T_j^m\rangle \} = \langle \alpha_j^m \rangle
\end{align}
 for $j=1,2$.
\end{corollary}

The las result was obtained in  \cite[Theorem 1.3 and Corollary 4.2]{Lu-Darvas-DiNezza-mono} when $m=n$. 

\begin{proof}[Proof of Corollary  \ref{cor-Tcongvoitheta2}] 
At this point, the proof is similar to that of \cite[Corollary 4.2]{Lu-Darvas-DiNezza-mono}. The implication from (\ref{eq-congu1u2theta}) to (\ref{eq-congu1u2thetatach}) is clear thanks to the monotonicity. Assume now (\ref{eq-congu1u2thetatach}).

Let $\omega'$ be a big enough K\"ahler form such that $\{\omega'\}+ \alpha_j$ is K\"ahler for $j=1,2$. Using this, the fact that $\alpha_j$ is big nef, and Theorem  \ref{the-congthemomega}, we deduce that 
$$\{\langle (T_j+\omega')^m\rangle \} = (\alpha_j + \{\omega'\})^m $$
for $j=1,2$. Now the convexity of the class of currents of full mass intersection in K\"ahler cohomology classes (see, for example, \cite[Theorem 1.3]{Viet-generalized-nonpluri} or \cite{DiNezza-stability}) gives that 
$$\big\{ \big\langle (T_1+T_2+ 2 \omega')^m \big\rangle \big\} = (\alpha_1 +\alpha_2+ 2\{\omega'\})^m.$$
By this and Theorem  \ref{the-congthemomega} again, we get the desired equality. The proof is finished.
\end{proof}

\section{Proof of Theorem \ref{the-self-intersec}} \label{sec-the1}

In this section, we will give a proof of Theorem \ref{the-self-intersec}.   
The following result implies Theorem \ref{the-self-intersec} in the case where $\alpha$ is K\"ahler.

\begin{theorem}\label{the-kahlerself-intersec} Let $\cali{K}_0$ be a compact subset in the K\"ahler cone of $(1,1)$-classes on $X$. Let $\alpha\in \cali{K}_0$ and  let $T$ be a closed positive current in $\alpha$. Let $1\le m \le n$ be an integer.  Let $\cali{V}$ be  the set of  maximal irreducible analytic subsets $V$ of dimension at least $n-m$ of $X$ such that the generic Lelong number $\nu(T,V)$ of $T$ along $V$ is strictly positive.
 Then, we have  
\begin{align}\label{ine-obstructionselfinterkahler}
\big\| \alpha^m - \{\langle T^m \rangle \}  \big\| \ge C\sum_{V \in \cali{V}}\big(\nu(T,V)\big)^{n- \dim V} \, \vol(V),
\end{align}
for some constant $C>0$ independent of $\alpha$ (but depending on $\cali{K}_0$). 
\end{theorem}

By the proof below, one can see that if the sum in the right-hand of  (\ref{ine-obstructionselfinterkahler}) is taken only on $V \in \cali{V}$ such that the dimension of $V$ is equal to $n-m$, then the factor $C$ can be replaced by  $1$.

\proof By Demailly's analytic approximation of psh functions (see \cite{Demailly_analyticmethod}), it is enough to check (\ref{ine-obstructionselfinterkahler}) when $T$ has analytic singularities (note that the maximality of $V$ is preserved here). Hence, the polar locus $I_T$ of $T$ is an analytic subset on $X$. 

Let $0 \le l \le n-1$ be an integer. Let $\cali{V}_l$ be the subset of $\cali{V}$ consisting of $V$ of dimension $l$. 
Let $V \in \cali{V}_l$. By the definition of $\cali{V}$, we have $l \ge n-m$. Put $S:= \langle T^{n-l-1}\rangle$. Note that $\langle T^{n-l} \rangle = \langle T \dot{\wedge} S \rangle$ by Proposition \ref{pro-sublinearnonpluripolar}.  Using this and  monotonicity of non-pluripolar products, we get 
\begin{align}\label{ine-chuyenveVdimllossmass}
\big\| \alpha^m - \{\langle T^m \rangle \}  \big\| \ge \big\| \alpha^m - \{\langle T^{n-l} \rangle \wedge \alpha^{l+m-n}\}\big\| \gtrsim  \big\| \alpha^{n-l} - \{\langle T \dot{\wedge}S \rangle\}\big\|.
\end{align}

Let $1 \le s \le n-l-1$ be an integer. Since $V$ is of dimension $l$ and $V$ is a maximal irreducible analytic subset of $X$ such that $\nu(T,V)>0$, we see that the self-intersection $T^{s}$ is classically well-defined in an open neighborhood $U$ of $V \backslash \Sing(I_T)$ in $X$, where $\Sing (I_T)$ is the singular part of the analytic set $I_T$. Moreover, $T^s$ has no mass on $V$ because $s < n-l$. By \cite[Page 169]{Demailly_analyticmethod} or Corollary \ref{cor-sosanhlelong}, we have 
\begin{align}\label{ine-TswelldefinedU}
\nu(T^s, V\cap U) \ge  \nu(T,V)^s.
\end{align}
Now using the fact that $T$ has analytic singularities and the maximality of $V$, we get 
$$\langle T^s \rangle =\bold{1}_{U \backslash V} T^s = T^s$$
on $U$ (we choose $U$ such that $U$ doesn't intersect the singularity of $I_T$). Using these observations for $s=n-l-1$ gives
\begin{align}\label{ine-TswelldefinedU2}
\nu(S, V) \ge  \nu(T,V)^{n-l-1}.
\end{align}
Now recall that non-pluripolar products have no mass on pluripolar sets. In particular, $S$ has no mass on $I_T$.
This allows us to apply Lemma \ref{le-truonghom=1lelong} to $T$ and $S$. As a result, we obtain 
$$\big \| \alpha \wedge \{S\} - \{\langle T \dot{\wedge}S \rangle \}\big\| \ge \sum_{V \in \cali{V}_l} \nu(S,V) \nu(T,V) \vol(V) \ge \sum_{V \in \cali{V}_l} \nu(T,V)^{n-l} \vol(V)$$
by (\ref{ine-TswelldefinedU2}). On the other hand, by monotonicity of non-pluripolar products again, we get $\{S\} \le \alpha^{n-l-1}$. This combined with the fact that $\alpha$ is K\"ahler implies $\alpha \wedge \{S\} \le \alpha^{n-l}$. Hence, we deduce that 
$$\big \| \alpha^{n-l} - \{\langle T \dot{\wedge}S \rangle \}\big\| \ge \sum_{V \in \cali{V}_l} \nu(T,V)^{n-l} \vol(V).$$
This coupled with (\ref{ine-chuyenveVdimllossmass}) gives (\ref{ine-obstructionselfinterkahler}). The proof is finished.
\endproof

We now prove Theorem \ref{the-self-intersec}. Let the notations be as in the statement of that result. 

\begin{proof}[End of the proof of Theorem \ref{the-self-intersec}]   Let $T_{\min}$ be a current with minimal singularities in $\alpha$. Let $\theta$ be a smooth  closed $(1,1)$-form such that $\beta:= \alpha+ \{\theta\}$ is K\"ahler. We can choose $\theta$ to depend only on $\|\alpha\|$.   Applying Theorem \ref{the-congthemomega} to $T, T_{\min}, \theta$ gives 
$$ \big \| \{\langle (T_{\min}+\theta)^m \rangle\} -\{\langle (T+\theta)^m \rangle\} \big\| \le  C\big\|  \alpha^m- \{\langle T^m \rangle\}\big\|^{2^{-m}},$$
for some constant $C$ depending only on $\cali{N}_0$. Using this and  the fact that $\alpha$ is big nef, we infer
$$ \big \|  \beta^m  -\{\langle (T+\theta)^m \rangle\} \big\| \le  C\big\|  \alpha^m- \{\langle T^m \rangle\}\big\|^{2^{-m}}.$$
Now applying Theorem \ref{the-kahlerself-intersec} to $\beta, T+\theta$ shows that the left-hand side of the last inequality is greater than or equal to 
$$\frac{1}{m}\sum_{V \in \cali{V}}\big(\nu(T,V)\big)^{n- \dim V} \, \vol(V).$$
Hence (\ref{ine-obstructionselfinter}) follows. The proof is complete.
\end{proof}

\bibliography{biblio_family_MA,biblio_Viet_papers}
\bibliographystyle{siam}

\bigskip

\noindent
\Addresses
\end{document}